\documentclass[10pt]{article}
\usepackage[a4paper, tmargin=3cm, bmargin=3cm, lmargin=2.75cm, rmargin=2.75cm, textheight=24cm, textwidth=16cm]{geometry}
\usepackage{setspace}
\usepackage{amsfonts}
\usepackage{graphicx} 
\usepackage{latexsym}
\usepackage{amsthm}
\usepackage{amsmath}
\usepackage{amssymb}
\usepackage{array}
\usepackage{enumerate}
\usepackage{enumitem}
\usepackage{breqn}
\usepackage{tikz}
\usepackage{thmtools, thm-restate}

\newtheorem{theorem}{Theorem}[section]

\newtheorem{example}{Example}[section]

\newtheorem{note}{Note}[section]

\newlist{notes}{enumerate}{1}
\setlist[notes]{label=Note: , leftmargin=*}

\title{On the Spectral Analysis of Power Graph of Dihedral Groups}
\author{
    Basit Auyoob Mir, Fouzul Atik,  Priti Prasanna Mondal, \\
    \small Department of Mathematics, SRM University-AP, Andhra Pradesh 522240, India. \\
    \small \texttt{mirbasit553@gmail.com, fouzulatik@gmail.com, pritiprasanna1992@gmail.com}
}
\date{}

\begin{document}
\maketitle

\begin{abstract}
\noindent
The power graph \( \mathcal{G}_G \) of a group \( G \) is a graph whose vertex set is \( G \), and two elements \( x, y \in G \) are adjacent if one is an integral power of the other. In this paper, we determine the adjacency, Laplacian, and signless Laplacian spectra of the power graph of the dihedral group \( D_{2pq} \), where \( p \) and \( q \) are distinct primes. Our findings demonstrate that the results of Romdhini et al. [2024], published in the \textit{European Journal of Pure and Applied Mathematics}, do not hold universally for all \( n \geq 3 \). Our analysis demonstrates that their results hold true exclusively when \( n = p^m \) where \( p \) is a prime number and \( m \) is a positive integer. The research examines their methodology via explicit counterexamples to expose its boundaries and establish corrected results. This study improves past research by expanding the spectrum evaluation of power graphs linked to dihedral groups.
\\

\end{abstract}

\section{Introduction}
Modern research on studying algebra through graph-theoretic examination has produced diverse intriguing structures that connect algebraic components. The power graph of a group (denoted as $\mathcal{G}_G$) represents one of the notable constructions that researchers now focus on intensively. Given a group $G$, the power graph $\mathcal{G}_G$ is a graph whose vertex set is $G$, with two distinct vertices $x$ and $y$ being adjacent if and only if one is a power of the other, i.e., $x = y^k$ or $y = x^k$ for some integer $k$.  
The dihedral group \cite{rose2009} $D_{2n}$ presents an interesting power graph structure since this is a non-abelian group of order $2n$ consisting of generators $a$ and $b$ where $a$ has order $n$ and $b$ has order $2$ satisfying the relation $ba = a^{-1}b = a^{n-1}b$. The spectral properties of $\mathcal{G}_{D_{2n}}$ enable researchers to understand the graph structure and the underlying group by studying its adjacency spectrum, its Laplacian spectrum, and its signless Laplacian spectrum.

The adjacency matrix \cite{brouwer2011} of power graph of $D_{2n}$, denoted by $A(\mathcal{G}_{D_{2n}})$, is defined as the $2n \times 2n$ matrix whose $(i, j)$-th entry is 1 if $i\sim j$  and  is $0$ otherwise. The diagonal degree matrix \cite{brouwer2011}, $D(\mathcal{G}_{D_{2n}})$, is another kind of matrix and is defined as the $2n \times 2n$ matrix where each diagonal entry represents the degree of the corresponding vertex. Another matrix associated with the graphs is the Laplacian matrices denoted as $L(\mathcal{G}_{D_{2n}})$, which totally depends on the adjacency and diagonal degree matrices, and is defined as follows \cite{brouwer2011} 
$L(\mathcal{G}_{D_{2n}}) = D(\mathcal{G}_{D_{2n}}) - A(\mathcal{G}_{D_{2n}})
$ and the signless Laplacian matrix \cite{brouwer2011} 
$\mathcal{SL}(\mathcal{G}_{D_{2n}}) = D(\mathcal{G}_{D_{2n}}) + A(\mathcal{G}_{D_{2n}})$.
On performing the eigenvalue analysis of such matrices, they provide information on power graph connection status, classification attributes, and structural aspects, which are used to identify spectral features of the power graph.  The insights obtained from these matrix spectra push researchers to reexamine group algebraic features that arise in the underlying network structure.

The investigation of power graphs has a rich history. Kelarev and Quinn  \cite{kelarev2000, kelarev2002} first introduced the concept of a power graph in the context of semigroups. Researchers studied power graphs more extensively throughout different group structures starting from cyclic groups \cite{chattopadhyay2019,kumar2021}, moving to dihedral groups \cite{romdhini2024}, and extending the analysis to broad group categories. These studies have focused on diverse aspects, such as determining graph-theoretic properties like connectedness, diameter, and clique number and analyzing spectral characteristics. For instance, several studies have characterized the power graphs of specific groups and established relationships between group properties and graph parameters \cite{moghaddamfar2014, cameron2020}. Sriparna et al. \cite{chattopadhyay2018} initiated the study of the power graphs of groups in which they investigated the spectral properties of power graphs for cyclic, dihedral, and dicyclic groups.  They derive bounds for the spectral radii and partially determine the spectra of these group power graphs. Mehranian et al. in \cite{mehranian2017} analyzed the spectra of power graphs for various groups, including cyclic, dihedral, elementary abelian groups of prime power order, and the Mathieu group. Their findings offer key spectral insights into the relationship between algebraic and graph-theoretic structures. In \cite{panda2019}, various aspects of the Laplacian spectra of power graphs of finite cyclic, dicyclic, and finite \( p \)-groups have been studied, and the algebraic connectivity has been completely determined for finite \( p \)-groups. Additionally, in \cite{panda2019}, the multiplicity of the Laplacian spectral radius has been analyzed, providing complete results for dicyclic and finite \( p \)-groups. Also, in \cite{chattopadhyay2015}, the Laplacian spectrum of the power graph of the additive cyclic group \( \mathbb{Z}_n \) and the dihedral group \( D_n \) has been studied. It is shown that the Laplacian spectrum of \( \mathcal{G}(D_n) \) is the union of that of \( \mathcal{G}(\mathbb{Z}_n) \) and \(\{2n,1\}\). Additionally, the algebraic connectivity of \( \mathcal{G}(D_n) \) is determined, and bounds for the same are provided for \( \mathcal{G}(\mathbb{Z}_n) \). The spectral analysis of power graphs, in particular, has become an active area of research, with investigations into the spectra of power graphs of finite groups leading to interesting connections between algebraic and combinatorial properties. For more on the developments of power graphs of some finite groups, we refer the reader to \cite{kumar2021}. 

Recently, Romdhini et al. \cite{romdhini2024} published a paper in the \textit{European Journal of Pure and Applied Mathematics} focusing on the spectral properties of the power graph of dihedral groups, denoted by $D_{2n}$. They presented formulations for the characteristic polynomials of these graphs and claimed their results held generically. However, our analysis reveals that their claimed genericity is not universally valid. This paper critically examines the findings presented in \cite{romdhini2024}. We demonstrate, through counterexamples, that their main result concerning the characteristic polynomials of the power graph of dihedral groups does not hold for all $n \geq 3$. Specifically, we show that their results are only valid when $n$ is of the form $p^m$, where $p$ is a prime number and $m$ is a positive integer. Furthermore, we provide a correction to their result and explicitly calculate the adjacency, Laplacian, and signless Laplacian spectrum of the power graph of $D_{2pq}$ where $p$ and $q$ are distinct primes.

Throughout this paper, we extend the study of power graph spectra by addressing gaps in existing results. Section 2 presents essential definitions and preliminaries. In Section 3, we provide counterexamples to results proved in \cite{romdhini2024}. Finally, in Section 4, we determine the full adjacency, Laplacian, and signless Laplacian spectrum of $\mathcal{G}_{D_{2pq}}$, further enriching the spectral analysis of power graphs.

\section{Preliminaries}
\noindent This section presents some of the basic definitions together with key statements and main theorems which establish our research foundation before introducing our main results. These initial findings are derived from the existing research studies. This work provides major insights into spectral and structural properties that were not there in the research of power graphs before and related algebraic structures. Subsequent theorems and assumptions operate as essential foundations for our exploration. The subsequent sections use these results to build their foundations.



\begin{theorem}
    \emph{\cite{chakrabarty2009}} Let \( G \) be a finite group. The \( \mathcal{G}_G \) is complete if and only if \( G \) is either a cyclic group of order \( 1 \) or a cyclic group of order \( p^m \) for some prime \( p \) and for some \( m \in \mathbb{N} \).  
\end{theorem}
\begin{theorem}
 \emph{\cite{mehranian2017}} Let $G=D_{2n}$ be a dihedral group of order $2n$. Let \( n \) be a prime power. Then the characteristic polynomials of both the power graph of the dihedral group \( D_{2n} \) can be determined as follows:
\[
P(\mathcal{G}(G))=\lambda^{n-1}(\lambda+1)^{n-2}(\lambda^3-(n-2)\lambda^2-(2n-1)\lambda+n^2-2n)
\]
\end{theorem}


\begin{theorem} \label{theo2.5}\emph{\cite{chattopadhyay2015}} For any non-prime positive integer \( n > 3 \), the Laplacian eigenvalues of \( \mathcal{G}_{D_n} \) can be expressed in terms of those of \(\mathcal{G}_{\mathbb{Z}_n} \) as follows:
\[
\lambda_i(\mathcal{G}_{D_n}) =
\begin{cases}
    2n, & \text{if } i = 1 \\
    \lambda_i(\mathcal{G}_{\mathbb{Z}_n}) = n, & \text{for } 2 \leq i \leq \phi(n) + 1 \\
    \lambda_i(\mathcal{G}_{\mathbb{Z}_n}), & \text{for } \phi(n) + 2 \leq i \leq n - 1 \\
    1, & \text{for } n \leq i \leq 2n - 1 \\
    0, & \text{if } i = 2n
\end{cases}
\]
\end{theorem}

\begin{theorem}\label{upperblockmat}\emph{\cite{horn2012}} Let \( B \) be a block upper triangular matrix of the form

\[
B =
\begin{bmatrix}
B_{11} & B_{12} & \cdots & B_{1k} \\
0 & B_{22} & \cdots & B_{2k} \\
\vdots & \vdots & \ddots & \vdots \\
0 & 0 & \cdots & B_{kk}
\end{bmatrix}
\]

where each \( B_{ii} \) is a square matrix. Then, the determinant of \( B \) is given by

\[
\det(B) = \det(B_{11}) \det(B_{22}) \cdots \det(B_{kk}).
\]
\end{theorem}

\section{Counterexample and Spectral Properties}
The dihedral group $D_{12}=\{\langle a, b \rangle ~|~
(i)~ a^{6}=b^{2}=e ~(ii) ~ba=a^{-1}b=a^{5}b\}$. \\
i.e., $D_{12}=\{ e, a, a^2, a^3, a^4, a^5, b, ab, a^2b, a^3b, a^4b, a^5b \}$. The power graph of the group $D_{12}$ is shown in Figure \ref{fig1}. 
The eigenvalues for the power graph $D_{12}$ are calculated as follows:

\begin{figure}[h]%
\centering
\includegraphics[width=0.3\textwidth]{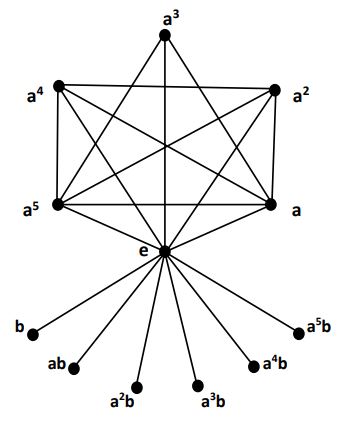}
\caption{Power graph of $D_{12}$}
\label{fig1}
\end{figure}

\begin{example}
\label{ex1}
\emph{In this example, we have calculated all the adjacency eigenvalues of the power graph $D_{12}$. And then we compared it with the given characteristic polynomial in \cite{romdhini2024}. We write the adjacency matrix with row and column index by the vertex set $\{ e, a, a^2, a^3, a^4, a^5, b, ab, a^2b, a^3b, a^4b, a^5b \}$, making an ordering with partition $\{V_{1}, V_{2}, V_{3}, V_{4}, V_{5}\}$. Where $V_{1}=\{e\}, V_{2}=\{a, a^5\}, V_{3}=\{a^2, a^4\}, V_{4}=\{a^4\}, V_{5}=\{b, ab, a^2b, a^3b, a^4b, a^5b\} $. }
\begin{eqnarray*}
    A(\mathcal{G}_{D_{12}})=
    \left[
    \begin{array}{c|c|c|c|c}
    0 & J_{1\times 2} &  J_{1\times 2} & J_{1} & J_{1\times 6}  \\
    \hline
    J_{2\times 1} & (J-I)_{2} & J_{2} & J_{2\times 1} & O_{2\times 6} \\ 
    \hline
    J_{2\times 1} & J_{2\times 1} & (J-I)_{2\times 2} & O_{2\times 1} & J_{2\times 6} \\
    \hline
    J_{1} & J_{1\times 2} & O_{1\times 2} & 3I_{1} & O_{1\times 6} \\
    \hline
    J_{6\times 1} & O_{6\times 2} & O_{6\times 2} & O_{6\times 1} & O_{6\times 6}
    \end{array}
    \right].
\end{eqnarray*}

\emph{The eigenvalues of $ A(\mathcal{G}_{D_{12}})$ are $0^{(5)}, -1^{(2)},-2.924,-1.647,0.356,1.480,4.735.$
But according to the article \cite{romdhini2024}, the characteristic polynomial of $A(\mathcal{G}_{{D}_{12}})$ is $P_{A(\mathcal{G}_{{D}_{12}})}(\lambda)= \lambda^{5}(\lambda+1)^{4}(\lambda^{3}-4\lambda^{2}-11\lambda+24)$. According to this characteristic polynomial, the eigenvalues are $0^{(5)}, -1^{(4)}, -2.84198, 1.61589, 5.22609.$ This clearly shows that these are not the eigenvalues of the power graph of the dihedral group ${D}_{12}$.}
\end{example}

\begin{example}
\emph{Similarly, for this example, we follow the same vertex ordering and partition as in Example \ref{ex1}. Then we calculated all the Laplacian eigenvalues of the power graph $D_{12}$. And then we compared it with the given characteristic polynomial in \cite{romdhini2024}.}
\begin{eqnarray*}
    \mathcal{L}(\mathcal{G}_{D_{12}})=
    \left[
    \begin{array}{c|c|c|c|c}
    11 & -J_{1\times 2} &  -J_{1\times 2} & -J_{1} & -J_{1\times 6}  \\
    \hline
    -J_{2\times 1} & (6I-J)_{2} & -J_{2} & -J_{2\times 1} & O_{2\times 6} \\ 
    \hline
    -J_{2\times 1} & -J_{2\times 1} & (5I-J)_{2\times 2} & O_{2\times 1} & -J_{2\times 6} \\
    \hline
    -J_{1} & -J_{1\times 2} & O_{1\times 2} & (4I-J)_{1} & O_{1\times 6} \\
    \hline
    -J_{6\times 1} & O_{6\times 2} & O_{6\times 2} & O_{6\times 1} & I_{6\times 6}
    \end{array}
    \right].
\end{eqnarray*}
\emph{ The spectrum of $\mathcal{L}(\mathcal{G}_{D_{12}})$ are ${0, 1^{(6)}, 6^{(2)}, 5, 3, 12}$. However, the Laplacian characteristic polynomial presented in \cite{romdhini2024}, is $P_{\mathcal{L}(\mathcal{G}_{{D}_{12}})}(\lambda)= \lambda(\lambda-12)(\lambda-6)^{4}(\lambda-1)^{6}$ which is not correct.}
\end{example}

\begin{example}
\emph{Similarly, we have calculated all the signless laplacian eigenvalues of the power graph $D_{12}$, by following the same vertex ordering and partition as in Example \ref{ex1}. And then we compared it with the given characteristic polynomial in \cite{romdhini2024}.}
\begin{eqnarray*}
    \mathcal{SL}(\mathcal{G}_{D_{12}})=
    \left[
    \begin{array}{c|c|c|c|c}
    11 & J_{1\times 2} &  J_{1\times 2} & J_{1} & J_{1\times 6}  \\
    \hline
    J_{2\times 1} & (4I+J)_{2} & J_{2} & J_{2\times 1} & O_{2\times 6} \\ 
    \hline
    J_{2\times 1} & J_{2\times 1} & (3I+J)_{2\times 2} & O_{2\times 1} & J_{2\times 6} \\
    \hline
    J_{1} & J_{1\times 2} & O_{1\times 2} & (2I+J)_{1} & O_{1\times 6} \\
    \hline
    J_{6\times 1} & O_{6\times 2} & O_{6\times 2} & O_{6\times 1} & I_{6\times 6}
    \end{array}
    \right].
\end{eqnarray*}
\emph{ The characteristic polynomial of the signless Laplacian matrix of $\mathcal{G}_{D_{12}}$ is
$P_{\mathcal{SL}(\mathcal{G}_{{D}_{12}})}(\lambda)=(\lambda-1)^5(\lambda-4)^2(\lambda-3)(\lambda^4-22\lambda^3+137\lambda^2-236\lambda+72)$.
But the signless Laplacian characteristic polynomial presented in \cite{romdhini2024} is
$P_{\mathcal{SL}(\mathcal{G}_{{D}_{12}})}(\lambda)=(\lambda-1)^5(\lambda-4)^4(\lambda-3)(\lambda^3-21\lambda^2+108\lambda-40)$,
which is wrong.}
\end{example}


\section{Main Results}
In this section, we discussed the spectra of power graphs of the dihedral group $D_{2n}$ where $n$ is a product of two distinct primes.

In the following theorem, we provide all the adjacency eigenvalues (or characteristic polynomial) of the power graph of the dihedral group $D_{2pq}$ of order $2pq$, where $p$ and $q$ are distinct primes. 

\begin{theorem}
\label{Adj_D2pq}
Consider the dihedral group $G=D_{2pq}$ of order $2pq$, the spectrum of the power graph $\mathcal{G}_{D_{2pq}}$ are $0^{(pq-1)}$, $-1^{(pq-4)}$ and rest of the spectrum are 
 the roots of the equation
\begin{align*}
\lambda^5 + (4 - pq)\lambda^4 + (7 - pq - p - q)\lambda^3
+ M\lambda^2 
+ N\lambda+K=0
\end{align*}
Where $M=2p^2q^2 - 2p^2q - 2pq^2 + p^2 + q^2 - 5p - 5q + 8; 
~N=2p^2q^2 - p^2q + p^2 - pq^2 - pq - 4p + q^2 - 4q + 4; \\
K=2p^3q^2-p^3q^3-p^3q+2p^2q^3-4p^2q^2+3p^2q-pq^3+3pq^2-4pq$
\end{theorem}

\begin{proof}
It is obvious that the vertex set of the $\mathcal{G}_{D_{2pq}}$ will be $\{e, a, a^2, ..., a^{pq-1}, b, ab, a^{2}b, ..., a^{pq-1}b\}$. Now, let's partition these vertices as $\{V_{1},V_{2},V_{3},V_{4},V_{5}\}$ such that $V_1$ contains the identity element of $D_{2pq}$ and $V_2$ contain all the generators that are $V_2=\{a^i:\gcd(i,pq)=1\}$ so that $|V_1|=1$ and order of $V_2=\phi(pq)=(p-1)(q-1)$, also $V_3=\{p,2p,...,(q-1)p\}$, $V_4=\{q,2q,...,(p-1)q\}$ and $V_5=\{b,ab,...,a^{pq-1}b\}$. Then we get the adjacency matrix, in which rows and columns are indexed by the vertex set using the partition $\{V_{1}, V_{2}, V_{3}, V_{4}, V_{5}\}$.

\begin{equation}
A(\mathcal{G}_{D_{2pq}})=
\begin{bmatrix}
O_1 & J_{1\times\phi(pq)} & J_{1\times(q-1)} & J_{1\times(p-1)} & J_{1\times pq}\\
J_{\phi(pq)\times 1} & (J-I)_{\phi(pq)} & J_{\phi(pq)\times (q-1)} & J_{\phi(pq)\times(p-1)} & O_{\phi(pq)\times pq}\\
J_{(q-1)\times 1} & J_{(q-1)\times\phi(pq)} & (J-I)_{(q-1)} & O_{(q-1)\times(p-1)} & O_{(q-1)\times pq}\\
J_{(p-1)\times 1} & J_{(p-1)\times\phi(pq)} & O_{(p-1)\times(q-1)} & (J-I)_{(p-1)} & O_{(p-1)\times pq}\\
J_{pq\times 1} & O_{pq\times\phi(pq)} & O_{pq\times(q-1)} & O_{pq\times(p-1)} & O_{pq\times pq} 
\end{bmatrix}
\label{eq:einstein}
\end{equation}
Now, implement the following steps to find the characteristic polynomial of the matrix $A(\mathcal{G}_{D_{2pq}})-\lambda I_{2pq}$.  
\begin{enumerate}
     \item $R_{pq+i} \to R_{pq+i} - R_{pq+1}$  for \( 2 \leq i \leq 2pq \).
    \item  $C_{pq+1}  \to   C_{pq+1} + (C_{pq+2} + \cdots + C_{2pq})$.
\end{enumerate}
So the characteristic polynomial (in terms of determinant form) is

\begin{eqnarray*}
    P_{A(\mathcal{G}_{D_{2pq}})}(\lambda)=\left|
    \begin{array}{ccccccccccc|cccc}
                -\lambda & 1 & \cdots & 1  &  1 & \cdots & 1 & 1 & \cdots & 1 & pq & 1 & \cdots & 1 \\
            
                1 & -\lambda & \cdots & 1  &  1 & \cdots & 1 & 1 & \cdots & 1 & 0 & 0 & \cdots & 0 \\
                \vdots & \vdots  & \ddots & \vdots  &  \vdots & \ddots & \vdots & \vdots & \ddots & \vdots & \vdots & \vdots & \ddots & \vdots \\
                1 & 1 &  \cdots & -\lambda  &  1 & \cdots & 1 & 1 & \cdots & 1 & 0 & 0 & \cdots & 0 \\
              
                1 & 1 & \cdots & 1  &  -\lambda & \cdots & 1 & 0 & \cdots & 0 & 0 & 0 & \cdots & 0 \\ 
                \vdots & \vdots & \ddots & \vdots  &  \vdots & \ddots & \vdots & \vdots & \ddots & \vdots & \vdots  & \vdots & \ddots & \vdots \\
                1 & 1 & \cdots & 1  &  1 & \cdots & -\lambda & 0 & \cdots & 0 & 0 & 0 & \cdots & 0 \\
            
              1 & 1 & \cdots & 1  &  0 & \cdots & 0 & -\lambda & \cdots & 1 & 0 & 0 & \cdots & 0  \\
              \vdots & \vdots & \ddots & \vdots  &  \vdots & \ddots & \vdots & \vdots & \ddots & \vdots & \vdots & \vdots & \ddots & \vdots \\
                 1 & 1 & \cdots & 1  &  0 & \cdots & 0 & 1 & \cdots & -\lambda & 0 & 0 & \cdots & 0 \\
                 
                1 & 0 & \cdots & 0  &  0 & \cdots & 0 & 0 & \cdots & 0 & -\lambda & 0 & \cdots & 0 \\
                \hline
                0 & 0 & \cdots & 0  &  0 & \cdots & 0 & 0 & \cdots & 0 & 0 & -\lambda & \cdots & 0 \\
                \vdots & \vdots & \ddots & \vdots  &  \vdots & \ddots & \vdots & \vdots & \ddots & \vdots & \vdots & \vdots & \ddots & \vdots \\
                0 & 0 & \cdots & 0  &  0 & \cdots & 0 & 0 & \cdots & 0 & 0 & 0 & \cdots & -\lambda \\
                \end{array}
                \right|.
\end{eqnarray*} 
$=\left|
\begin{array}{cc}
R & Q \\
O & S
\end{array}
\right|$

Where, 
$R = \left[
\begin{array}{ccccccccccc}
-\lambda & 1 & \cdots & 1 & 1 & \cdots & 1 & 1 & \cdots & 1 & pq \\
1 & -\lambda & \cdots & 1 & 1 & \cdots & 1 & 1 & \cdots & 1 & 0 \\
\vdots & \vdots & \ddots & \vdots & \vdots & \ddots & \vdots & \vdots & \ddots & \vdots & \vdots \\
1 & 1 & \cdots & -\lambda & 1 & \cdots & 1 & 1 & \cdots & 1 & 0 \\
1 & 1 & \cdots & 1 & -\lambda & \cdots & 1 & 0 & \cdots & 0 & 0 \\
\vdots & \vdots & \ddots & \vdots & \vdots & \ddots & \vdots & \vdots & \ddots & \vdots & \vdots \\
1 & 1 & \cdots & 1 & 1 & \cdots & -\lambda & 0 & \cdots & 0 & 0 \\
1 & 1 & \cdots & 1 & 0 & \cdots & 0 & -\lambda & \cdots & 1 & 0 \\
\vdots & \vdots & \ddots & \vdots & \vdots & \ddots & \vdots & \vdots & \ddots & \vdots & \vdots \\
1 & 1 & \cdots & 1 & 0 & \cdots & 0 & 1 & \cdots & -\lambda & 0 \\
1 & 0 & \cdots & 0 & 0 & \cdots & 0 & 0 & \cdots & 0 & -\lambda \\
\end{array}
\right]$,
$Q=\left[ 
\begin{array}{ccc}
1 & \cdots & 1 \\
0 & \cdots & 0 \\
\vdots & \ddots & \vdots \\
0 & \cdots & 0 \\
0 & \cdots & 0 \\
\vdots & \ddots & \vdots \\
0 & \cdots & 0 \\
0 & \cdots & 0 \\
\vdots & \ddots & \vdots \\
0 & \cdots & 0 \\
0 & \cdots & 0 \\
\end{array}
\right]$

$O=\left[
\begin{array}{ccccccccccc}
0 & 0 & \cdots & 0 & 0 & \cdots & 0 & 0 & \cdots & 0 & 0\\
\vdots & \vdots & \ddots & \vdots & \vdots & \ddots & \vdots & \vdots & \ddots & \vdots & \vdots\\
0 & 0 & \cdots & 0 & 0 & \cdots & 0 & 0 & \cdots & 0 & 0\\
\end{array}
\right],$  
$S=\left[ 
\begin{array}{ccc}
-\lambda & \cdots & 0 \\
\vdots & \ddots & \vdots \\
0 & \cdots & -\lambda \\
\end{array}
\right]$

By applying Theorem \ref{upperblockmat}, we get  \[ P_{A(\mathcal{G}_{D_{2pq}})}(\lambda)=\det(R)\det(S)
\]
Now, we know that the $\det(S)=(-\lambda)^{pq-1}$. So, to find the characteristic polynomial, we need to solve the determinant of $R$.

Now apply the following steps on $R$
\begin{enumerate}
\item $R_i\to R_i-R_2$, $3\leq i\leq \phi(pq)+1$
\item $R_j\to R_j-R_{\phi(pq)+2}$, $\phi(pq)+3\leq j\leq \phi(pq)+q$
\item $R_k\to R_k-R_{\phi(pq)+q+1}$, $\phi(pq)+q+2\leq k\leq pq$
\end{enumerate}
After row transformation, we will apply column transformation as
\begin{enumerate}
    \item $C_2\to C_2+C_3+\cdots+C_{\phi(pq)+1}$
\item $C_{\phi(pq)+2}\to C_{\phi(pq)+2}+C_{\phi(pq)+3}+\cdots+C_{\phi(pq)+q}$
\item $C_{\phi(pq)+q+1}\to C_{\phi(pq)+q+1}+C_{\phi(pq)+q+2}+\cdots+C_{pq}$.
\end{enumerate}
Then we get 
\[
\det(R)=(\lambda+1)^{pq-4}
\begin{vmatrix}
-\lambda & (p-1)(q-1) & q-1 & p-1 & pq \\
1 & (p-1)(q-1)-1-\lambda & q-1 & p-1 & 0 \\
1 & (p-1)(q-1) & q-2-\lambda & 0 & 0 \\
1 & (p-1)(q-1) & 0 & (p-2)-\lambda & 0 \\
1 & 0 & 0 & 0 & -\lambda
\end{vmatrix}
\]

\begin{align*}
=(\lambda+1)^{pq-4}\{\lambda^5 + (4 - pq)\lambda^4 + (7 - pq - p - q)\lambda^3
+ M\lambda^2 
+ N\lambda+K\}
\end{align*}
Where $M=2p^2q^2 - 2p^2q - 2pq^2 + p^2 + q^2 - 5p - 5q + 8; 
~N=2p^2q^2 - p^2q + p^2 - pq^2 - pq - 4p + q^2 - 4q + 4; \\
K=2p^3q^2-p^3q^3-p^3q+2p^2q^3-4p^2q^2+3p^2q-pq^3+3pq^2-4pq$

Therefore, the characteristic polynomial is 
$$P_{A(\mathcal{G}_{D_{2pq}})}(\lambda)
=(-\lambda)^{pq-1}(\lambda+1)^{pq-4}\{\lambda^5 + (4 - pq)\lambda^4 + (7 - pq - p - q)\lambda^3
+ M\lambda^2 
+ N\lambda+K\}$$

Where $M=2p^2q^2 - 2p^2q - 2pq^2 + p^2 + q^2 - 5p - 5q + 8; 
~N=2p^2q^2 - p^2q + p^2 - pq^2 - pq - 4p + q^2 - 4q + 4; \\
K=2p^3q^2-p^3q^3-p^3q+2p^2q^3-4p^2q^2+3p^2q-pq^3+3pq^2-4pq$

\end{proof}

The next theorem provides all the laplacian eigenvalues (or characteristic polynomial) of the power graph of the dihedral group $D_{2pq}$ of order $2pq$, where $p$ and $q$ are distinct primes.
\begin{theorem}
Let $G = D_{2pq}$ be the dihedral group of order $2pq$, then the spectrum of  Laplacian of $\mathcal{G}_{D_{2pq}}$ are
$$0,1^{(pq)},pq^{(\phi(pq))},(pq-p+1)^{(q-2)},(pq-q+1)^{(p-2)},pq-p-q+2,2pq$$.
\end{theorem}

\begin{proof}
The construction of the Laplacian matrix of $\mathcal{G}_{D_{2pq}}$ is based on its degree and adjacency matrices. Using the same vertex ordering and partition as in Theorem \ref{Adj_D2pq}, the degree matrix of $\mathcal{G}_{D_{2pq}}$ is
\begin{equation}
\mathcal{D}(\mathcal{G}_{D_{2pq}})=
\begin{bmatrix}
2pq-1 & O_{1\times\phi(pq)} & O_{1\times(q-1)} & O_{1\times(p-1)} & O_{1\times pq}\\
O_{\phi(pq)\times 1} & (pq-1)I_{\phi(pq)} & O_{\phi(pq)\times (q-1)} & O_{\phi(pq)\times(p-1)} & O_{\phi(pq)\times pq}\\
O_{(q-1)\times 1} & O_{(q-1)\times\phi(pq)} & (pq-p)I_{(q-1)} & O_{(q-1)\times(p-1)} & O_{(q-1)\times pq}\\
O_{(p-1)\times 1} & O_{(p-1)\times\phi(pq)} & O_{(p-1)\times(q-1)} & (pq-q)I_{(p-1)} & O_{(p-1)\times pq}\\
O_{pq\times 1} & O_{pq\times\phi(pq)} & O_{pq\times(q-1)} & O_{pq\times(p-1)} & I_{pq\times pq} 
\end{bmatrix}
\end{equation}
From the definition of the Laplacian matrix of \( \mathcal{G}(D_{2pq}) \), it follows that  
\[
\mathcal{L}(\mathcal{G}_{D_{2pq}}) = \mathcal{D}(\mathcal{G}_{D_{2pq}}) - A(\mathcal{G}_{D_{2pq}})
\]
where \( \mathcal{D}(\mathcal{G}_{D_{2pq}}) \) represents the degree matrix and \( A(\mathcal{G}_{D_{2pq}}) \) denotes the adjacency matrix of \( \mathcal{G}_{D_{2pq}} \).
Therefore, 
\[
\mathcal{L}(\mathcal{G}_{D_{2pq}}) =
\small
\begin{bmatrix}
2pq-1 & -J_{1\times\phi(pq)} & -J_{1\times(q-1)} & -J_{1\times(p-1)} & -J_{1\times pq}\\
-J_{\phi(pq)\times 1} & (pq I-J)_{\phi(pq)} & -J_{\phi(pq)\times (q-1)} & -J_{\phi(pq)\times(p-1)} & O_{\phi(pq)\times pq}\\
-J_{(q-1)\times 1} & -J_{(q-1)\times\phi(pq)} & ((pq-p+1)I-J)_{(q-1)} & O_{(q-1)\times(p-1)} & O_{(q-1)\times pq}\\
-J_{(p-1)\times 1} & -J_{(p-1)\times\phi(pq)} & O_{(p-1)\times(q-1)} & ((pq-q+1)I-J)_{(p-1)} & O_{(p-1)\times pq}\\
J_{pq\times 1} & O_{pq\times\phi(pq)} & O_{pq\times(q-1)} & O_{pq\times(p-1)} & I_{pq\times pq} 
\end{bmatrix}.
\]
Now, perform the following transformations on the matrix $\mathcal{L}(\mathcal{G}_{D_{2pq}})- \lambda I_{2pq}$, to compute the characteristic polynomial:
\begin{enumerate}
    \item $R_{pq+i} \to R_{pq+i} - R_{pq+1}$  for \( 2 \leq i \leq 2pq \).
    \item  $C_{pq+1}  \to   C_{pq+1} + (C_{pq+2} + \cdots + C_{2pq})$.
\end{enumerate}
After this, we will have a $2\times2$ block upper triangular matrix (a similar kind of matrix in Theorem \ref{Adj_D2pq}).
Again, by applying the row transformation on the 1st block to this Laplacian matrix
\begin{enumerate}
\item $R_i\to R_i-R_2$, $3\leq i\leq \phi(pq)+1$
\item $R_j\to R_j-R_{\phi(pq)+2}$, $\phi(pq)+3\leq j\leq \phi(pq)+q$
\item $R_k\to R_k-R_{\phi(pq)+q+1}$, $\phi(pq)+q+2\leq k\leq pq$
\end{enumerate}
After this, on applying the column transformation, we have
\begin{enumerate}
\item $C_2\to C_2+C_3+\cdots+C_{\phi(pq)+1}$
\item $C_{\phi(pq)+2}\to C_{\phi(pq)+2}+C_{\phi(pq)+3}+\cdots+C_{\phi(pq)+q}$
\item $C_{\phi(pq)+q+1}\to C_{\phi(pq)+q+1}+C_{\phi(pq)+q+2}+\cdots+C_{pq}$  
\end{enumerate}

We will get the following:

\[P_{\mathcal{L}(\mathcal{G}_{D_{2pq}})}(\lambda) =
(\lambda-1)^{pq-1}(\lambda-(pq-q+1))^{p-2}(\lambda-(pq-p+1))^{q-2}(\lambda-pq)^{\phi(pq)-1}det(B) .\]
Where
\[B=
    \begin{pmatrix}
2pq-1-\lambda & -(p-1)(q-1) & 1-q & 1-p & -pq \\
-1 & p+q-1-\lambda & 1-q & 1-p & 0 \\
-1 & -(p-1)(q-1) & pq-p-q+2-\lambda & 0 & 0 \\
-1 & -(p-1)(q-1) & 0 & pq-p-q+2-\lambda & 0 \\
-1 & 0 & 0 & 0 & 1-\lambda
\end{pmatrix}.
\]
Again applying some rows and columns operations, one can get \[det(M)=(-\lambda)(\lambda-1)(\lambda-pq)(\lambda-pq+p+q-2)(\lambda-2pq).\]
Finally, \[P_{\mathcal{L}(\mathcal{G}_{D_{2pq}})}(\lambda) =(-\lambda)
(\lambda-1)^{pq}(\lambda-pq)^{\phi(pq)}(\lambda-(pq-p+1))^{q-2}(\lambda-(pq-q+1))^{p-2}(\lambda-(pq-p-q+2))(\lambda-2pq).\]

\end{proof}
\begin{note}
   \emph{ The above result can be found in \cite{chattopadhyay2015} as a corollary of the theorem \ref{theo2.5}. However, our approach is different and more convenient.}
\end{note}

Lastly, we provide all the signless Laplacian eigenvalues (or characteristic polynomial) of the power graph of the dihedral group $D_{2pq}$ of order $2pq$, where $p$ and $q$ are distinct primes.

\begin{theorem}
Let $G = D_{2pq}$ be the dihedral group of order $2pq$, then the spectrum of signless laplacian of $\mathcal{G}_{D_{2pq}}$ are
$1^{(pq-1)}, (pq-2)^{\phi(pq)}, (pq-p-1)^{(q-2)}, (pq-q-1)^{(p-2)}$
and the rest of the spectrum are the roots of the biquadratic equation
\begin{align*}
\lambda^4 -(5pq-p-q-3)\lambda^3 -X\lambda^2 -Y\lambda -Z = 0
\end{align*}
Where $ X=-8p^2q^2+4p^2q+4pq^2+5pq+p+q-4; ~Y=4p^3q^3-4p^3q^2-4p^2q^3+4p^2q^2-2p^2q-2pq^2-2pq+4p+4q; ~Z=-2p^3q^3+2p^3q^2+2p^2q^2-2p^2q-2pq^2-4p-4q+8.$
\end{theorem}

\begin{proof}
From the definition of the Signless Laplacian matrix of \( \mathcal{G}_{D_{2pq}}\), it follows that  
\[
\mathcal{SL}(\mathcal{G}_{D_{2pq}}) = \mathcal{D}(\mathcal{G}_{D_{2pq}})+A(\mathcal{G}_{D_{2pq}})
\]
where \( \mathcal{D}(\mathcal{G}_{D_{2pq}}) \) represents the degree matrix and \( A(\mathcal{G}_{D_{2pq}}) \) denotes the adjacency matrix of \( \mathcal{G}_{D_{2pq}} \).
Therefore, 
\[
\mathcal{SL}(\mathcal{G}_{D_{2pq}})= 
\begin{bmatrix}
2pq-1 & J_{1\times\phi(pq)} & J_{1\times(q-1)} & J_{1\times(p-1)} & J_{1\times pq}\\
J_{\phi(pq)\times 1} & ((pq-2)I+J)_{\phi(pq)} & J_{\phi(pq)\times (q-1)} & J_{\phi(pq)\times(p-1)} & O_{\phi(pq)\times pq}\\
J_{(q-1)\times 1} & J_{(q-1)\times\phi(pq)} & ((pq-p-1)I+J)_{(q-1)} & O_{(q-1)\times(p-1)} & O_{(q-1)\times pq}\\
J_{(p-1)\times 1} & J_{(p-1)\times\phi(pq)} & O_{(p-1)\times(q-1)} & ((pq-q-1)I+J)_{(p-1)} & O_{(p-1)\times pq}\\
J_{pq\times 1} & O_{pq\times\phi(pq)} & O_{pq\times(q-1)} & O_{pq\times(p-1)} & I_{pq\times pq} 
\end{bmatrix}
\]
Similarly, execute the following operations on the matrix $\mathcal{SL}(\mathcal{G}_{D_{2pq}})-\lambda I_{2pq}$, to compute the characteristic polynomial:
\begin{enumerate}
    \item $R_{pq+i} \to R_{pq+i} - R_{pq+1}$  for \( 2 \leq i \leq 2pq \).
    \item  $C_{pq+1}  \to   C_{pq+1} + (C_{pq+2} + \cdots + C_{2pq})$.
\end{enumerate}
After this, we will have a $2\times2$ block upper triangular matrix (a similar kind of matrix in Theorem \ref{Adj_D2pq}). Again, the following row transformation on the 1st block is applied to this signless Laplacian matrix.
\begin{enumerate}
\item $R_i\to R_i-R_2$, $3\leq i\leq \phi(pq)+1$
\item $R_j\to R_j-R_{\phi(pq)+2}$, $\phi(pq)+3\leq j\leq \phi(pq)+q$
\item $R_k\to R_k-R_{\phi(pq)+q+1}$, $\phi(pq)+q+2\leq k\leq pq$
\end{enumerate}
After this, on applying the column transformation, we have
\begin{enumerate}
\item $C_2\to C_2+C_3+\cdots+C_{\phi(pq)+1}$
\item $C_{\phi(pq)+2}\to C_{\phi(pq)+2}+C_{\phi(pq)+3}+\cdots+C_{\phi(pq)+q}$
\item $C_{\phi(pq)+q+1}\to C_{\phi(pq)+q+1}+C_{\phi(pq)+q+2}+\cdots+C_{pq}$  
\end{enumerate}
We get the following format of the characteristic polynomial:
\[P_{\mathcal{SL}(\mathcal{G}_{D_{2pq}})}(\lambda) =
(\lambda-1)^{pq-1}(\lambda-(pq-q-1))^{p-2}(\lambda-(pq-p-1))^{q-2}(\lambda-(pq-2))^{\phi(pq)-1}det(C).\]
Where
\[C=
    \begin{pmatrix}
2pq-1-\lambda & (p-1)(q-1) & q-1 & p-1 & pq \\
1 & 2pq-p-q-1-\lambda & q-1 & p-1 & 0 \\
1 & (p-1)(q-1) & pq-p+q-2-\lambda & 0 & 0 \\
1 & (p-1)(q-1) & 0 & pq-q+p-2-\lambda & 0 \\
1 & 0 & 0 & 0 & 1-\lambda
\end{pmatrix}.
\]
After computing the determinant, we get \[det(N)=(\lambda-(pq-2))(-\lambda^4 +(5pq-p-q-3)\lambda^3 +X\lambda^2 +Y\lambda +Z).\]
Where $ X=-8p^2q^2+4p^2q+4pq^2+5pq+p+q-4; ~Y=4p^3q^3-4p^3q^2-4p^2q^3+4p^2q^2-2p^2q-2pq^2-2pq+4p+4q; ~Z=-2p^3q^3+2p^3q^2+2p^2q^2-2p^2q-2pq^2-4p-4q+8.$

Therefore, \[P_{\mathcal{SL}(\mathcal{G}_{D_{2pq}})}(\lambda) =
(\lambda-1)^{pq-1}(\lambda-(pq-q-1))^{p-2}(\lambda-(pq-p-1))^{q-2}(\lambda-(pq-2))^{\phi(pq)}(-\lambda^4 +(5pq-p-q-3)\lambda^3 +X\lambda^2 +Y\lambda +Z).\]
Where $ X=-8p^2q^2+4p^2q+4pq^2+5pq+p+q-4; ~Y=4p^3q^3-4p^3q^2-4p^2q^3+4p^2q^2-2p^2q-2pq^2-2pq+4p+4q; ~Z=-2p^3q^3+2p^3q^2+2p^2q^2-2p^2q-2pq^2-4p-4q+8.$

\end{proof}

\begin{note}
    \emph{One can find the above result in \cite{banerjee2020a}, which is a lengthy and complicated proof. Whereas our method is very short and easier.}
\end{note}

\section*{Conclusion}  
In this paper, we have investigated the adjacency, Laplacian, and signless Laplacian spectra of the power graph \(\mathcal{G}_{D_{2pq}}\) for the dihedral group \(D_{2pq}\), where \(p\) and \(q\) are distinct primes. Our results correct and refine the work of Romdhini et al. \cite{romdhini2024}, demonstrating that their spectral formulas hold only when \(n = p^m\) for a prime \(p\) and \(m \in \mathbb{N}\). Through explicit counterexamples and a detailed analysis of the structure of \(\mathcal{G}_{D_{2pq}}\), we established complete spectral characterizations, filling a gap in the existing literature.  

The methods developed here—particularly the block matrix decomposition and eigenvalue analysis—can be extended to other classes of finite groups, such as dicyclic or symmetric groups. Future work may explore:  
\begin{itemize}  
    \item Spectral properties of power graphs for non-abelian groups with more complex orders (e.g., \(D_{2p^2q}\)).  
    \item Applications of these spectra in network theory, such as studying synchronization dynamics or robustness.  
    \item Connections between the algebraic structure of \(G\) and the automorphism group of \(\mathcal{G}_G\).  
\end{itemize}  

\noindent  
Our findings underscore the deep relationship between group theory and spectral graph theory, paving the way for further interdisciplinary research.  


\section*{Declarations}
\textbf{Conflict of interest:} There are no conflicts of interest, according to the authors.

\end{document}